\documentclass[12pt]{amsart}

\usepackage{ucs}

\usepackage{amssymb}
\usepackage{amsthm}
\usepackage{amsmath}
\usepackage{latexsym}
\usepackage[cp1251]{inputenc}
\usepackage[mathcal]{eucal}
\usepackage{graphicx}
\usepackage{wrapfig}
\usepackage{caption}
\usepackage{subcaption}
\usepackage{indentfirst}
\usepackage[left=3cm,right=3cm,top=3cm,bottom=3cm,bindingoffset=0cm]{geometry}
\usepackage{enumerate}

\DeclareMathOperator{\aut}{Aut}

\DeclareMathOperator{\diag}{Diag}

\DeclareMathOperator{\cay}{Cay}
\DeclareMathOperator{\cyc}{Cyc}

\DeclareMathOperator{\iso}{Iso}

\DeclareMathOperator{\orb}{Orb}

\DeclareMathOperator{\pr}{pr}

\DeclareMathOperator{\rk}{rk}

\DeclareMathOperator{\Span}{Span}

\DeclareMathOperator{\sym}{Sym}
\DeclareMathOperator{\rad}{rad}

\makeatletter 
\def\@seccntformat#1{\csname the#1\endcsname. } 
\def\@biblabel#1{#1.}

\makeatother

\title{On separability of Schur rings over abelian $p$-groups}
\author{G. K. Ryabov}
\address{Novosibirsk State University, 2 Pirogova St., 630090 Novosibirsk, Russia}
\email{gric2ryabov@gmail.com}

\thanks{\rm The work is supported by the Russian Foundation for Basic Research (project 17-51-53007)}

\date{}

\newtheorem{prop}{Proposition}[section]

\newtheorem*{theo1}{Theorem 1}
\newtheorem*{theo2}{Theorem 2}

\newtheorem{lemm}{Lemma}[section]

\theoremstyle{definition}

\begin{document}

\vspace{\baselineskip}
\vspace{\baselineskip}

\vspace{\baselineskip}

\vspace{\baselineskip}

\begin{abstract}
An $S$-ring (Schur ring) is called \emph{separable} with respect to a class of $S$-rings $\mathcal{K}$ if it is determined up to isomorphism in  $\mathcal{K}$ only by the tensor of its structure constants. An abelian group is said to be  \emph{separable} if every $S$-ring over this group is separable with respect to the class of $S$-rings over abelian groups. Let $C_n$ be a cyclic group of order~$n$ and $G$ be a noncylic abelian $p$-group. From the previously obtained results it follows that if $G$ is separable then $G$ is isomorphic to $C_p\times C_{p^k}$ or $C_p\times C_p\times C_{p^k}$, where $p\in \{2,3\}$ and $k\geq 1$. We prove that the groups $D=C_p\times C_{p^k}$ are separable whenever $p\in \{2,3\}$. From this statement we deduce that a given Cayley graph over  $D$ and a given Cayley graph over an arbitrary abelian group one can check whether these graphs are isomorphic in time $|D|^{O(1)}$.
\end{abstract}

\maketitle

Let $G$ be a finite group. A subring of the group ring  $\mathbb{Z}G$ is called an \emph{$S$-ring} (a \emph{Schur ring}) over $G$ if it is a free $\mathbb{Z}$-module spanned on a special partition of $G$ (exact definitions are given in Section~$1$). The elements of this partition  are called the \emph{basic sets} of the $S$-ring and the number of the basic sets is called the \emph{rank} of the $S$-ring. The notion of the  $S$-ring goes back to Schur and Wielandt. They used  ``the $S$-ring method''  to study a permutation group having  a regular subgroup \cite{Schur,Wi}.

Let $\mathcal{A}$ and $\mathcal{A}^{'}$ be $S$-rings over groups $G$  and  $G^{'}$ respectively. \emph{A (combinatorial) isomorphism } from $\mathcal{A}$ to $\mathcal{A}^{'}$  is defined to be a bijection  $f:G\rightarrow G^{'}$ such that for every basic set $X$ of $\mathcal{A}$ the set $X^{'}=f(X)$ is a basic set of $\mathcal{A}^{'}$ and $f$ is an isomorphism of the Cayley graphs $\cay(G,X)$ and 
$\cay(G^{'},X^{'})$.  \emph{An algebraic isomorphism} from  $\mathcal{A}$ to $\mathcal{A}^{'}$ is defined to be a ring isomorphism of them  inducing the bijection between the basic sets of $\mathcal{A}$ and  the basic sets of $\mathcal{A}^{'}$. It can be checked that every combinatorial isomorphism induces an algebraic isomorphism. If every algebraic isomorphism from~$\mathcal{A}$ to another $S$-ring is induced  by a combinatorial isomorphism then $\mathcal{A}$ is said to be \emph{separable}. More precisely, we say that  $\mathcal{A}$ is \emph{separable} with respect to a class of $S$-rings if every algebraic isomorphism from $\mathcal{A}$ to an $S$-ring from this class is induced by a combinatorial isomorphism. Every separable $S$-ring is determined up to isomorphism only by the tensor of its structure constants. For more details see  \cite{EP3,EP7}.

Denote the classes of $S$-rings over cyclic and abelian groups by  $\mathcal{K}_C$ and  $\mathcal{K}_A$ respectively. A cyclic group of order $n$ is denoted by  $C_n$.  It was proved in \cite{EP3} that every  $S$-ring over a cyclic  $p$-group is separable with respect to $\mathcal{K}_C$. Infinite series of $S$-rings over cyclic groups that are nonseparable with respect to $\mathcal{K}_C$ were constructed in \cite{EP7}.

An abelian group $G$ is said to be \emph{separable}  if every  $S$-ring over $G$ is separable with respect to  $\mathcal{K}_A$. One can prove that a subgroup of a separable group is separable. \cite[Section~$3$]{Klin} implies that the group $H\times H$ is nonseparable  for every group $H$ of order at least~$4$. Let $G$ be a noncyclic abelian $p$-group. It follows that if $G$ is separable then  $G$ is isomorphic to $C_p\times C_{p^k}$ or $C_p\times C_p\times C_{p^k}$, where $p\in \{2,3\}$ and $k\geq 1$. In the present paper we prove that the groups from the first family  are separable. We will consider the question on a separability of the groups from the second family in a more general context in the next paper.

\begin{theo1}\label{main}
The group $C_p\times C_{p^k}$ is separable for $p\in \{2,3\}$ and $k\geq 1$.
\end{theo1}

The proof of Theorem~$1$ is based on the description of all  $S$-rings over the group~$D$ from this theorem that was obtained for  $p=2$ in \cite{MP3} and for $p=3$ in \cite{Ry}. If $\mathcal{A}$ is an $S$-ring over $D$ of rank at least~$3$ then one of the following statements holds: $(1)$~$\mathcal{A}$ is the tensor  product or the generalized wreath product of two  smaller  $S$-rings; $(2)$~$\mathcal{A}$ is \emph{cyclotomic} (it means that $\mathcal{A}$ is determined by a suitable subgroup of $\aut(D)$). The detailed description of   $S$-rings over $D$ is given in Lemma~\ref{element} and Lemma~\ref{Sring}. The separability of tensor products and generalized wreath products follows from the separability of operands. The most difficult task here is to check that cyclotomic  $S$-rings  are separable (see Section~$7$).

There is a relationship between the separability of $S$-rings and the isomorphism problem for Cayley graphs (see \cite[Section~$6.2$]{EP5}). In the case when all  $S$-rings over a  group of order~$n$ are separable the isomorphism problem for Cayley graphs over this group can be solved in time $n^{O(1)}$ by using the Weisfeiler-Leman algorithm  \cite{Weis,WeisL}. By Theorem~$1$ this implies  (see Section~$8$) the following statement. 

\begin{theo2}\label{main2}
Suppose that the group $D\cong C_p\times C_{p^k}$, where $p\in \{2,3\}$ and $k\geq 1$, is given explicitly. Then  for every Cayley graph  $\Gamma$ over  $D$ and  every Cayley graph $\Gamma^{'}$ over an arbitrary explicitly given  abelian group one can check in time  $|D|^{O(1)}$  whether $\Gamma$ and $\Gamma^{'}$ are isomorphic.
\end{theo2}

It should be mentioned that the isomorphism problem for Cayley graphs over cyclic groups was solved    in \cite{EP6} and \cite{M} independently.

In Sections $1$-$3$ we recall some definitions and facts concerned with $S$-rings and Cayley schemes. The most part of them is taken from \cite{MP2}.

The author would like to thank prof. I. Ponomarenko and prof. A. Vasil'ev for  their constructive comments which  help us to improve the text significantly.

\section{$S$-rings and Cayley schemes}

Let $G$ be a finite group and $\mathbb{Z}G$ be the integer group ring. Denote the identity element of $G$ by $e$. If $X\subseteq G$ then  the formal sum $\sum_{x\in X} {x}$ is considered as an element of $\mathbb{Z}G$ and denoted by $\underline{X}$. The set $\{x^{-1}:x\in X\}$ is denoted by $X^{-1}$.
	
A subring  $\mathcal{A}\subseteq \mathbb{Z} G$ is called an \emph{$S$-ring} over $G$ if there exists a partition $\mathcal{S}=\mathcal{S}(\mathcal{A})$ of~$G$ such that:

$1)$ $\{e\}\in\mathcal{S}$,

$2)$  if $X\in\mathcal{S}$ then $X^{-1}\in\mathcal{S}$,

$3)$ $\mathcal{A}=\Span_{\mathbb{Z}}\{\underline{X}:\ X\in\mathcal{S}\}$.

The elements of $\mathcal{S}$ are called the \emph{basic sets} of  $\mathcal{A}$ and the number $|\mathcal{S}|$ is called the \emph{rank} of  $\mathcal{A}$. If $X,Y,Z\in\mathcal{S}$ then   the number of distinct representations of $z\in Z$ in the form $z=xy$ with $x\in X$ and $y\in Y$ is denoted by $c^Z_{X,Y}$. Note that if $X$ and $Y$ are basic sets of $\mathcal{A}$ then $\underline{X}~\underline{Y}=\sum_{Z\in \mathcal{S}(\mathcal{A})}c^Z_{X,Y}\underline{Z}$. Therefore the numbers  $c^Z_{X,Y}$ are structure constants of $\mathcal{A}$ with respect to the basis $\{\underline{X}:\ X\in\mathcal{S}\}$.

Let $\mathcal{R}$ be a partition of $G\times G$. A pair $\mathcal{C}=\left(G,\mathcal{R}\right)$ is called a \emph{Cayley scheme} over $G$ if the following properties hold:

$1)$ $\diag(G\times G)=\{(g,g):g\in G\}\in\mathcal{R}$;

$2)$ if  $R\in\mathcal{R}$ then $R^*=\{(h,g): (g,h)\in R\}\in\mathcal{R}$;

$3)$ if $R,~S,~T\in\mathcal{R}$ then the number $c^T_{R,S}=|\{h\in G:(g,h)\in R,~(h,f)\in S\}|$ does not depend on the choice of  $(g,f)\in T$;

$4)$ $\{(hg,fg):(h,f)\in R\}=R$ for every $R\in\mathcal{R}$ and every $g\in G$.

The numbers $c^T_{R,S}$ are called the \emph{intersection numbers} of $\mathcal{C}$, the elements of $\mathcal{R}$ are  called the \emph{basic relations} of $\mathcal{C}$, and the number $|\mathcal{R}|$ is called the \emph{rank} of $\mathcal{C}$. If $R\in \mathcal{R}$ and $g\in G$ then the number $n(R)=\{h:(g,h)\in R\}$ does not depend on the choice of~$g$ and it is called the \emph{valency} of $R$. Note that $n(R)=c^S_{R,R^{*}}$, where $S=\diag(G\times G)$.

There is a one-to-one correspondence between  $S$-rings and Cayley schemes over $G$. Namely, if $\mathcal{A}$ is an $S$-ring over $G$ then the  pair $\mathcal{C}(\mathcal{A})=\left(G,\mathcal{R}(\mathcal{A})\right)$, where $\mathcal{R}(\mathcal{A})=\{R(X):X\in \mathcal{S}(\mathcal{A})\}$ and $R(X)=\{(g,xg): g\in G, x\in X\}$, is a Cayley scheme over~$G$. Conversely, if $\mathcal{C}=\left(G,\mathcal{R}\right)$ is a Cayley scheme over $G$ then $\mathcal{S}(\mathcal{C})=\{X(R):R\in \mathcal{R}\}$, where $X(R)=\{x\in G: (e,x) \in R\}\subseteq G$, is a partition of $G$ that defines the $S$-ring $\mathcal{A}(\mathcal{C})$ over $G$. If $\mathcal{A}$ is an $S$-ring and $\mathcal{C}(\mathcal{A})$ is the corresponding Cayley scheme then $$c^{R(Z)}_{R(X),R(Y)}=c^Z_{X,Y}~\eqno(1)$$
for all $X,Y,Z\in \mathcal{S}(\mathcal{A})$.

\section{Isomorphisms}

Let  $\mathcal{A}$  and $\mathcal{A}^{'}$ be $S$-rings over groups $G$  and $G^{'}$ respectively and $\mathcal{C}=(G,\mathcal{R})$ and $\mathcal{C}^{'}=(G^{'},\mathcal{R}^{'})$ be Cayley schemes over $G$  and $G^{'}$ respectively. Set $\mathcal{S}=\mathcal{S}(\mathcal{A})$ and $\mathcal{S}^{'}=\mathcal{S}(\mathcal{A}^{'})$. A \emph{(combinatorial) isomorphism} from $\mathcal{C}$ to $\mathcal{C}^{'}$ is defined to be a bijection $f:G\rightarrow G^{'}$ such that $\mathcal{R}^{'}=\mathcal{R}^f$, where $\mathcal{R}^f=\{R^f:~R\in\mathcal{R}\}$ and $R^f=\{(g^f,~h^f):~(g,~h)\in R\}$. A \emph{(combinatorial) isomorphism} from $\mathcal{A}$  to $\mathcal{A}^{'}$ is defined to be a bijection $f:G\rightarrow G^{'}$  which is an isomorphism of the corresponding Cayley schemes  $\mathcal{C}(\mathcal{A})$ and $\mathcal{C}(\mathcal{A}^{'})$.

The group $\iso(\mathcal{A})$ of all isomorphisms from $\mathcal{A}$ onto itself has a normal subgroup
$$\aut(\mathcal{A})=\{f\in \iso(\mathcal{A}): R(X)^f=R(X)~\text{for every}~X\in \mathcal{S}(\mathcal{A})\}.$$
This subgroup is called the \emph{automorphism group} of $\mathcal{A}$ and denoted by $\aut(\mathcal{A})$; the elements of $\aut(\mathcal{A})$ are called \emph{automorphisms} of $\mathcal{A}$.

 An \emph{algebraic isomorphism} from $\mathcal{A}$  to $\mathcal{A}^{'}$ is defined to be a bijection $\varphi:\mathcal{S}\rightarrow\mathcal{S}^{'}$ such that $c_{X,Y}^Z=c_{X^{\varphi},Y^{\varphi}}^{Z^{\varphi}}$ for all $X,Y,Z\in \mathcal{S}$. The mapping $\underline{X}\rightarrow \underline{X}^{\varphi}$ is extended by linearity to the ring isomorphism of $\mathcal{A}$  and $\mathcal{A}^{'}$.  An \emph{algebraic isomorphism}  from  $\mathcal{C}$ to $\mathcal{C}^{'}$ is defined to be a bijection $\varphi:\mathcal{R}\rightarrow \mathcal{R}^{'}$ such  that $c_{R,S}^T=c_{R^{\varphi},S^{\varphi}}^{T^{\varphi}}$ for all $R,S,T\in \mathcal{R}$. If $\varphi$ is an algebraic isomorphism of  $\mathcal{A}$  and $\mathcal{A}^{'}$ then the mapping $R(X)\mapsto R(X^{\varphi})$ is an algebraic isomorphism of the corresponding Cayley schemes  $\mathcal{C}(\mathcal{A})$ and $\mathcal{C}(\mathcal{A}^{'})$ by $(1)$.

Every isomorphism $f$ of $S$-rings (Cayley schemes) preserves  the structure constants (intersection numbers) and hence $f$ induces the algebraic isomorphism $\varphi_f$. 

We say that a Cayley scheme $\mathcal{C}$ is \emph{separable} with respect to a class of Cayley schemes~$\mathcal{K}$ if for every $\mathcal{C}^{'}\in \mathcal{K}$ every algebraic isomorphism  $\varphi:\mathcal{C}\rightarrow \mathcal{C}^{'}$ is induced by an isomorphism. The next statement immediately follows from~$(1)$.

\begin{prop}\label{connect}
If $\mathcal{A}$ is an $S$-ring and $\mathcal{C}(\mathcal{A})$ is the corresponding Cayley scheme then 
$\mathcal{A}$ is separable with respect to a class of $S$-rings $\mathcal{K}$ if and only if $\mathcal{C}(\mathcal{A})$ is separable with respect to the class of Cayley schemes corresponding to $S$-rings from $\mathcal{K}$.  
\end{prop}

For a fixed algebraic isomorphism $\varphi$ from  $\mathcal{A}$ to $\mathcal{A}^{'}$ (from $\mathcal{C}$ to $\mathcal{C}^{'}$) the set  of all isomorphisms $f$ such that  $\varphi_f=\varphi$ is denoted by  $\iso(\mathcal{A},\mathcal{A}^{'},\varphi)$ ($\iso(\mathcal{C},\mathcal{C}^{'},\varphi)$) . If $H$ is a group then put $H_{right}=\{x\mapsto xh,~x\in H:h\in H\}$. From the definitions it follows that

$$G_{right}\iso(\mathcal{A},\mathcal{A}^{'},\varphi)G^{'}_{right}=\iso(\mathcal{A},\mathcal{A}^{'},\varphi).~\eqno(2)$$
Note that an $S$-ring $\mathcal{A}$  is separable with respect to a class of $S$-rings $\mathcal{K}$ if and only if 
$$\iso(\mathcal{A},\mathcal{A}^{'},\varphi)\neq \varnothing$$ 
for every  $\mathcal{A}^{'}\in \mathcal{K}$ and every algebraic isomorphism  $\varphi:\mathcal{A}\rightarrow \mathcal{A}^{'}$. The group ring $\mathbb{Z}G$ and the $S$-ring of  rank~$2$ over $G$ are separable with respect to the class of all  $S$-rings. In the former case every basic set is singleton and hence every algebraic isomorphism is induced by an isomorphism in a natural way. In the latter case  there exists the unique algebraic isomorphism from the $S$-ring of rank~$2$ over $G$ to the $S$-ring of rank~$2$ over a given  group and this algebraic isomorphism is induced by every isomorphism.

Another type of  isomorphism between  $S$-rings (Cayley schemes) arises from group isomorphism. A \emph{Cayley isomorphism} from  $\mathcal{A}$  to $\mathcal{A}^{'}$ (from $\mathcal{C}$ to $\mathcal{C}^{'}$)  is defined to be a group isomorphism $f:G\rightarrow G^{'}$ such that $\mathcal{S}^f=\mathcal{S}^{'}$ ($\mathcal{R}^f=\mathcal{R}^{'}$). If there exists a Cayley isomorphism from  $\mathcal{A}$  to $\mathcal{A}^{'}$ we write $\mathcal{A}\cong_{\cay}\mathcal{A}^{'}$. Every Cayley isomorphism is a (combinatorial) isomorphism, however the converse statement is not true.

\section{$S$-rings: basic facts and constructions}

Let $\mathcal{A}$ be an $S$-ring over a group $G$. A set $X \subseteq G$ is called an \emph{$\mathcal{A}$-set} if $\underline{X}\in \mathcal{A}$. A subgroup $H \leq G$ is called an \emph{$\mathcal{A}$-subgroup} if $H$ is an $\mathcal{A}$-set. Let $L \unlhd U\leq G$. A section $U/L$ is called an \emph{$\mathcal{A}$-section} if $U$ and $L$ are $\mathcal{A}$-subgroups. If $S=U/L$ is an $\mathcal{A}$-section then the module
$$\mathcal{A}_S=Span_{\mathbb{Z}}\left\{\underline{X}^{\pi}:~X\in\mathcal{S}(\mathcal{A}),~X\subseteq U\right\},$$
where $\pi:U\rightarrow U/L$ is the quotient homomorphism, is an $S$-ring over $S$.

If $X \subseteq G$ then the set
$$\rad(X)=\{g\in G:~Xg=gX=X\}$$ 
is a subgroup of $G$ and  it is called the \emph{radical} of $X$. If $X$ is an $\mathcal{A}$-set  then the groups $\langle X \rangle$ and $\rad(X)$ are $\mathcal{A}$-subgroups of~$G$.

Let  $\mathcal{A}$ and $\mathcal{A}^{'}$ be $S$-rings over  $G$ and $G^{'}$ respectively and $\varphi:\mathcal{A}\rightarrow \mathcal{A}^{'}$ be an algebraic isomorphism. It is easy to see that $\varphi$ is extended to a bijection between  $\mathcal{A}$- and $\mathcal{A}^{'}$-sets and hence between  $\mathcal{A}$- and $\mathcal{A}^{'}$-sections. The images of an $\mathcal{A}$-set $X$ and an $\mathcal{A}$-section $S$ under the action of $\varphi$ are denoted by $X^{\varphi}$ and $S^{\varphi}$ respectively. If $S$ is an $\mathcal{A}$-section then  $\varphi$ induces the algebraic isomorphism 
$$\varphi_S:\mathcal{A}_S\rightarrow \mathcal{A}^{'}_{S^{'}},$$
where $S^{'}=S^{\varphi}$. Since $c^{\{e\}}_{X,Y}=\delta_{Y,X^{-1}}|X|$ and $|X|=c^{\{e\}}_{X,X^{-1}}$ , where $\delta_{X,X^{-1}}$ is the Kronecker delta, we conclude that $(X^{-1})^{\varphi}=(X^{\varphi})^{-1}$ and $|X|=|X^{\varphi}|$ for every $\mathcal{A}$-set $X$. In particular, $|G|=|G^{'}|$. It can be checked that

$$ \langle X^{\varphi} \rangle = \langle X \rangle ^{\varphi},~ \rad(X^{\varphi})=\rad(X)^{\varphi}~\eqno(3)$$ for every  $\mathcal{A}$-set $X$ (see \cite[p.10]{EP4}).

We say that an $S$-ring $\mathcal{A}$ is \emph{symmetric} if $X=X^{-1}$ for every $X\in \mathcal{S}(\mathcal{A})$. Clearly that if $\mathcal{A}$ is symmetric and  $\varphi:\mathcal{A}\rightarrow \mathcal{A}^{'}$ is an algebraic isomorphism then  $\mathcal{A}^{'}$ is also symmetric.

If $X \subseteq G$ and $m\in \mathbb{Z}$ then the set $\{x^m: x \in X\}$ is denoted by $X^{(m)}$. Sets $X,Y\subseteq G$ are called \emph{rationally conjugate} if there exists $m\in \mathbb{Z}$ coprime to $|G|$ such that $Y=X^{(m)}$. Further we formulate two satements on  $S$-rings over abelian groups that were proved, in fact, by Schur in \cite{Schur}. We give these statements in the form which can be found in \cite[Section~$23$]{Wi}.

\begin{lemm} \label{burn}
Let $\mathcal{A}$ be an $S$-ring over an abelian group  $G$. Then  $X^{(m)}\in \mathcal{S}(\mathcal{A})$  for every  $X\in \mathcal{S}(\mathcal{A})$ and every  $m\in \mathbb{Z}$ coprime to $|G|$. Other words, the mapping $\sigma_m:g\mapsto g^m$ is a Cayley isomorphism from   $\mathcal{A}$ onto itself.
\end{lemm}

\begin{lemm} \label{sch}
Let $\mathcal{A}$ be an $S$-ring over an abelian group $G$, $p$ be a prime divisor of $|G|$, and  $H=\{g\in G:g^p=e\}$. Then for every  $\mathcal{A}$-set $X$ the set $X^{[p]}=\{x^p:x\in~X,~|X\cap Hx|\not\equiv 0\mod p\}$ is an $\mathcal{A}$-set.
\end{lemm}

	Let $K \leq \aut(G)$. Then the set $\orb(K, G)$  of all orbits    of  $K$ on $G$ forms a partition of  $G$ that defines an  $S$-ring $\mathcal{A}$ over $G$.  In this case  $\mathcal{A}$ is called \emph{cyclotomic} and denoted by $\cyc(K,G)$.

	Let  $\mathcal{A}_1$ and $\mathcal{A}_2$ be $S$-rings over $G_1$ and $G_2$ respectively. Then the set
	
	$$\mathcal{S}=\mathcal{S}(\mathcal{A}_1)\times \mathcal{S}(\mathcal{A}_2)=\{X_1\times X_2:~X_1\in \mathcal{S}(\mathcal{A}_1),~X_2\in \mathcal{S}(\mathcal{A}_2)\} $$
forms a partition of  $G=G_1\times G_2$ that defines an  $S$-ring over $G$. This $S$-ring is called the  \emph{tensor product}  of $\mathcal{A}_1$ and $\mathcal{A}_2$ and denoted by $\mathcal{A}_1 \otimes \mathcal{A}_2$.
	
Let $e_1$ and $e_2$ be the identity elements of $G_1$ and $G_2$ respectively. Then the set $\mathcal{S}=\mathcal{S}_1 \cup \mathcal{S}_2$, where
	
	$$\mathcal{S}_1=\{X_1\times \{e_2\}:~X_1\in \mathcal{S}(\mathcal{A}_1)\},~\mathcal{S}_2=\{G_1\times \{X_2\}:~X_2\in \mathcal{S}(\mathcal{A}_2)\setminus \{e_2\}\} $$
forms a partition of  $G=G_1\times G_2$ that defines an  $S$-ring over $G$. This $S$-ring is called the  \emph{wreath product}  of $\mathcal{A}_1$ and $\mathcal{A}_2$ and denoted by $\mathcal{A}_1 \wr \mathcal{A}_2$.

\begin{lemm} \label{schurtens}
	 $S$-rings $\mathcal{A}_1 \otimes \mathcal{A}_2$ and $\mathcal{A}_1 \wr \mathcal{A}_2$ are separable if and only if  $S$-rings  $\mathcal{A}_1$ and $\mathcal{A}_2$ are separable.
	\end{lemm}
	
\begin{proof}
Follows from~\cite[Theorem~1.20]{E}.
\end{proof}

\begin{lemm}{\em \cite[Lemma 2.1] {EP3}}\label{uniq}
Let $\mathcal{A}$ and $\mathcal{A}^{'}$ be $S$-rings over $G$ and $G^{'}$ respectively. Let $\mathcal{B}$ be the $S$-ring generated by $\mathcal{A}$ and an element $\xi\in \mathbb{Z}G$ and $\mathcal{B}^{'}$ be the $S$-ring generated by $\mathcal{A}^{'}$ and an element $\xi^{'}\in \mathbb{Z}G^{'}$.  Then given algebraic isomorphism  $\varphi:\mathcal{A}\rightarrow \mathcal{A}^{'}$ there is at most one algebraic isomorphism $\psi:\mathcal{B}\rightarrow \mathcal{B}^{'}$ extending $\varphi$ and such that $\psi(\xi)=\xi^{'}$.
\end{lemm}	
	
Following \cite{MP} we say that an $S$-ring $\mathcal{A}$  is \emph{quasi-thin} if $|X|\leq 2$ for every $X\in~\mathcal{S}(\mathcal{A})$.

\begin{lemm}\label{quasithin}
 Let $\mathcal{A}$ be a quasi-thin $S$-ring over  $G$. Suppose that there are no $\mathcal{A}$-subgroups $H$ in $G$ such that $H\cong C_2\times C_2,~ \mathcal{A}_H=\mathbb{Z}H$, and $\mathcal{A}_{G/H}=\mathbb{Z}(G/H)$. Then $\mathcal{A}$ is separable with respect to the class of all  $S$-rings.
\end{lemm}
 
\begin{proof}

If  $\mathcal{A}$ is quasi-thin then the corresponding Cayley scheme  $\mathcal{C}(\mathcal{A})$ is also quasi-thin, i.e. every basic relation of  $\mathcal{C}(\mathcal{A})$ has valency at most~$2$. From \cite[Theorem 1.1]{MP} it follows that every quasi-thin Cayley scheme which is not a Klein scheme (see \cite[p.2]{MP}), is separable with respect to the class of all Cayley schemes. If $\mathcal{C}(\mathcal{A})$ is a Klein scheme then there exists an $\mathcal{A}$-subgroup $H$ of $G$ such that $H\cong C_2\times C_2,~ \mathcal{A}_H=\mathbb{Z}H$, and $\mathcal{A}_{G/H}=\mathbb{Z}(G/H)$. However, this contradicts to the assumption of the lemma. Therefore $\mathcal{C}(\mathcal{A})$ and $\mathcal{A}$ are separable with respect to the class of all Cayley schemes and the class of all $S$-rings respectively.
\end{proof}

\section{Generalized wreath product}

	Let $\mathcal{A}$ be an  $S$-ring over $G$ and $U/L$ be an $\mathcal{A}$-section. We say that  $\mathcal{A}$ is the \emph{generalized wreath product} or \emph{$U/L$-wreath product} if $L \unlhd G$ and $L\leq \rad(X)$ for every $X\in \mathcal{S}(\mathcal{A})$ outside $U$. The generalized wreath product is  called \emph{nontrivial} or \emph{proper} if $L\neq 1$ and $U\neq G$. If $U=L$ then $\mathcal{A}$ coincides with the wreath product of  $\mathcal{A}_L$ and $\mathcal{A}_{G/L}$.

The main goal of this section is to prove a sufficient condition of separability for the generalized wreath product  over an abelian group. Before this we formulate some additional statements required for the proof.

If $f:G\rightarrow G^{'}$ is a bijection and $X\subseteq G$ then the restriction of  $f$  on $X$  is denoted by $f^X$. Let $H\leq G$. Given $X,Y\in G/H$  put 
$$G_{X\rightarrow Y}=\{f^X: f\in G_{right}, X^f=Y\}.$$ 

In the following three lemmas  $\mathcal{A}$ and $\mathcal{A}^{'}$ are $S$-rings over groups $G$ and $G^{'}$ respectively and $\varphi:\mathcal{A}\rightarrow \mathcal{A}^{'}$ is an algebraic isomorphism.

\begin{lemm}{\em\cite[Lemma 3.4]{EP4}}\label{simres}
If $f \in \iso(\mathcal{A},\mathcal{A}^{'},\varphi)$, $H$ is an $\mathcal{A}$-subgroup of $G$, and $H^{'}=H^{\varphi}$ then
$$hf^Xh^{'}\in \iso(\mathcal{A}_H,\mathcal{A}^{'}_{H^{'}},\varphi_H)$$
for all $X\in G/H,~h\in G_{H\rightarrow X}$, and $h^{'}\in G^{'}_{X^{'}\rightarrow H^{'}}$, where $X^{'}=X^f$.
\end{lemm}

\begin{lemm} {\em\cite[Theorem 3.3, $1$]{EP4}}\label{simgwr}
Let $G$ and $G^{'}$ be abelian and $U/L$ be an $\mathcal{A}$-section of $G$. Suppose that $\mathcal{A}$ is the $U/L$-wreath product, $U^{'}=U^{\varphi}$, and $L^{'}=L^{\varphi}$. Then $\mathcal{A}^{'}$ is the  $U^{'}/L^{'}$-wreath product.
\end{lemm}

\begin{lemm}{\em\cite[Theorem 3.5]{EP4}}\label{simgwr2}
In the conditions of Lemma~\ref{simgwr}  the set $\iso(\mathcal{A},\mathcal{A}^{'},\varphi)$ consists of all bijections  $f:G \rightarrow G^{'}$ possessing the following properties:

$1)$ $(G/U)^f=G^{'}/U^{'},~(G/L)^f=G^{'}/L^{'}$,  

$2)$ $f^{G/L}\in \iso(\mathcal{A}_{G/L},\mathcal{A}^{'}_{G^{'}/L^{'}},\varphi_{G/L})$,

$3)$ if $X\in G/U$ and $X^{'}=X^f$ then there exist $g\in G_{U\rightarrow X}$ and $g^{'}\in G^{'}_{X^{'}\rightarrow U^{'}}$ such that  $gf^Xg^{'}\in \iso(\mathcal{A}_{U},\mathcal{A}^{'}_{U^{'}},\varphi_{U})$.

\end{lemm}

\begin{lemm}\label{sepwr}
Let $\mathcal{A}$ be the $U/L$-wreath product over an abelian group $G$. Suppose that $\mathcal{A}_U$ and $\mathcal{A}_{G/L}$ are separable  with respect to $\mathcal{K}_A$ and $\aut(\mathcal{A}_U)^{U/L}=\aut(\mathcal{A}_{U/L})$. Then $\mathcal{A}$ is separable with respect to $\mathcal{K}_A$.
\end{lemm}

\begin{proof}
Let $\mathcal{A}^{'}$ be an $S$-ring over an abelian group  $G^{'}$, $\varphi:\mathcal{A}\rightarrow \mathcal{A}^{'}$ be an algebraic isomorphism, and $U^{'}=U^{\varphi},L^{'}=L^{\varphi}$. Then from Lemma~\ref{simgwr} it follows that $\mathcal{A}^{'}$ is the $U^{'}/L^{'}$-wreath product.  Since $\mathcal{A}_U$ and $\mathcal{A}_{G/L}$ are separable with respect to $\mathcal{K}_A$, the algebraic isomorphisms 
$$\varphi_U:\mathcal{A}_U\rightarrow \mathcal{A}^{'}_{U^{'}},~\varphi_{G/L}:\mathcal{A}_{G/L}\rightarrow \mathcal{A}^{'}_{G^{'}/L^{'}}$$
are induced by some isomorphisms  $f_1$ and $f_2$ respectively. Let $X\in G/U$. We can consider $X$ as a subset of  $G/L$ because $X$ is a union of some cosets of $L$. Put $X^{'}=X^{f_2}$. Choose $g\in G_{U\rightarrow X}$ and $g^{'}\in G^{'}_{X^{'}\rightarrow U^{'}}$. The bijection  $g^{U/L}f_2^{X/L}g^{'X^{'}/L^{'}}$ induces the algebraic isomorphism $\varphi_{U/L}$ by Lemma~\ref{simres} applied to $\mathcal{A}_{G/L}$-subgroup $U/L$ of $G/L$. Put
$$f_0=g^{U/L}f_2^{X/L}g^{'X^{'}/L^{'}}(f_1^{U/L})^{-1}.$$
Since $f_1^{U/L}$ also induces  $\varphi_{U/L}$, we conclude that $f_0\in \aut(\mathcal{A}_{U/L})$. There exists  $h_X\in \aut(\mathcal{A}_{U})$ such that $h_X^{U/L}=f_0$ because  $\aut(\mathcal{A}_{U/L})=\aut(\mathcal{A}_{U})^{U/L}$. Put
$$f_X=g^{-1}h_Xf_1(g^{'})^{-1}.$$ 
Let $f:G\rightarrow G^{'}$ be the bijection whose restriction on  $X$ coincides with $f_X$ for every $X\in G/U$. Let us check that  $f$ possesses Properties $1$-$3$ from Lemma~\ref{simgwr2}. It is clear that $(G/U)^f=G^{'}/U^{'}$ and $(G/L)^f=G^{'}/L^{'}$ and hence  $f$ possesses Property~$1$. By the definition of $f$ we have  $gf^Xg^{'}=gf_Xg^{'}=h_Xf_1$. So from~$(2)$ it follows that for every $X\in G/U$ the bijection  $gf^Xg^{'}$ induces the  algebraic isomorphism  $\varphi_U$. It proves that $f$ possesses Property~$3$. The straightforward computations show that 

$$f^{X/L}=(g^{-1}h_Xf_1(g^{'})^{-1})^{X/L}=(g^{U/L})^{-1}g^{U/L}f_2^{X/L}g^{'X^{'}/L^{'}}(f_1^{U/L})^{-1}f_1^{U/L}(g^{'X^{'}/L^{'}})^{-1}\\=f_2^{X/L}$$ 
for every $X\in G/U$. Therefore  $f^{G/L}=f_2$ and $f^{G/L}$ induces $\varphi_{G/L}$. So $f$ possesses Property~$2$. Thus $f\in \iso(\mathcal{A},\mathcal{A}^{'},\varphi)$ by Lemma~\ref{simgwr2}.  

It is worth noting that in the proof of this lemma we followed, in general, the scheme of the proof  of separability of coset $S$-rings  over cyclic groups (\cite[p.33]{EP4}).
\end{proof}

\section{$S$-rings over cyclic groups}

\begin{lemm}\label{simcycl}
Let  $G$ be a cyclic group of order  $n\neq 4$ and $\mathcal{A}$ be an $S$-ring over  $G$  such that  $\mathcal{A}=\mathbb{Z}G$ or $\mathcal{A}=\cyc(K,G)$, where $K=\{\varepsilon,\sigma\},~\sigma:x\rightarrow x^{-1}$. Suppose that $\varphi$ is an algebraic isomorphism from $\mathcal{A}$ to an $S$-ring $\mathcal{A}^{'}$ over an abelian group $G^{'}$. Then $G^{'}\cong G$.
\end{lemm}

\begin{proof}

From the properties of  an algebraic isomorphism it follows that  $|G|=|G^{'}|=n$. Suppose that $X\in \mathcal{S}(\mathcal{A})$ contains a generator of $G$.  Then $G^{'}=\langle X^{\varphi} \rangle$ by~$(3)$.  If $|X|=1$ then $|X^{\varphi}|=1$ and hence $G^{'}$ is cyclic. If $|X|=2$ then $X=X^{-1}$. So   $|X^{\varphi}|=2$ and $(X^{\varphi})^{-1}=X^{\varphi}$ by the properties of an algebraic isomorphism. Therefore either $X^{\varphi}=\{x,x^{-1}\}$ for some $x\in G^{'}$ or $X^{\varphi}=\{x,y\}$ for some  $x,y\in G^{'}$ such that $|x|=|y|=2$. In the former case $G^{'}$ is cyclic and hence it is isomorphic to $G$; in the latter case $|G|=|G^{'}|=4$ that contradicts to the assumption of the lemma.
\end{proof}

Let $\mathcal{A}$ be an $S$-ring over a cyclic group $G$. Put $\rad(\mathcal{A})=\rad(X)$, where $X$ is a basic set of $\mathcal{A}$ containing a generator of $G$. Note that $\rad(\mathcal{A})$ does not depend on the choice of $X$. Indeed, if $Y\in \mathcal{S}(\mathcal{A})$, $\langle Y \rangle =G$, and $Y\neq X$ then  $X$ and $Y$ are rationally conjugate by Theorem~\ref{burn} and hence $\rad(X)=\rad(Y)$.

\begin{lemm} \label{circrad}
Let $p$ be a prime and $\mathcal{A}$ be an $S$-ring over a cyclic  $p$-group~$G$. Suppose that $\rad(\mathcal{A}) > e$. Then there exists an  $\mathcal{A}$-section  $U/L$ such that $\mathcal{A}$ is the proper $U/L$-wreath product and $\rad(\mathcal{A}_U)=e$.
\end{lemm}

\begin{proof}
Let $X$ be the union of all basic sets of $\mathcal{A}$ with the trivial radical. Then $U=\langle X \rangle$ is a proper $\mathcal{A}$-subgroup and  $\rad(\mathcal{A}_U)=e$. There exists the least nontrivial $\mathcal{A}$-subgroup $L$ of $G$ because $G$ is a cyclic $p$-group. All basic sets of $\mathcal{A}$ outside  $U$ have nontrivial radical. Since the radical of every basic set is an $\mathcal{A}$-subgroup, we conclude that $L\leq \rad(X)$ for every $X\in \mathcal{S}(\mathcal{A})$ outside $U$. Thus $\mathcal{A}$ is the proper $U/L$-wreath product. 
\end{proof}

Let $p\in\{2,3\}$ and $k\geq 1$. Set $A=\langle a\rangle$ and $a_1=a^{p^{k-1}}$, where $|a|=p^k$. These notations are valid until the end of this section.

\begin{lemm}\label{circ}
Let  $\mathcal{A}$ be an $S$-ring over $A$. Suppose that $\rad(\mathcal{A})=e$. Then one of the following statements holds:

$1)$ $\rk(\mathcal{A})=2;$  

$2)$ $\mathcal{A}=\mathbb{Z}A;$

$3)$ $\mathcal{A}=\cyc(K,A)$, where $K=\{\varepsilon,\sigma\},~\sigma:x\rightarrow x^{-1};$

$4)$ $p=2$ and $\mathcal{A}=\cyc(K,A)$, where $K=\{\varepsilon,\sigma\},~\sigma:x\rightarrow a_1x^{-1}.$

In all cases $\mathcal{A}$ is separable with respect to the class of all  $S$-rings.

\end{lemm}

\begin{proof}

From  \cite[Theorem 4.1]{EP2} and \cite[Theorem 4.2]{EP2} it follows that every $S$-ring with the trivial radical over a cyclic group is the tensor product of cyclotomic $S$-rings with the trivial radical and  $S$-rings of  rank~$2$. Since  $A$ is a $p$-group, we conclude that either $\rk(\mathcal{A})=2$ or $\mathcal{A}=\cyc(K,A)$ for some $K\leq \aut(A)$. In the former case it is obvious that $\mathcal{A}$ is separable. In the latter case  \cite[Lemma 5.1]{EP} implies that one of  Statements
 $2$-$4$ holds. In particular, $\mathcal{A}$ is quasi-thin.  The group $A$ is cyclic and hence it does not contain  $\mathcal{A}$-subgroups $H$ such that $H\cong C_2\times C_2$. Therefore $\mathcal{A}$ is separable by Lemma~\ref{quasithin}.
 \end{proof}

Denote the symmetric group of a set $V$ by $\sym(V)$. If $\Gamma\leq \sym(V)$ then denote the set of all orbits of the componentwise action of $\Gamma$ on $V^2$ by $\orb(\Gamma,V^2)$. Permutation groups $\Gamma,~\Gamma^{'}\leq \sym(V)$   are called  $2$-\emph{equivalent} if  $\orb(\Gamma,V^2)=\orb(\Gamma^{'},V^2)$. A permutation group  $\Gamma\leq \sym(V)$ is called $2$-\emph{isolated} if  it is the only group which is $2$-equivalent to $\Gamma$.

\begin{lemm}\label{circaut}
Let $\mathcal{A}$ be the proper $U/L$-wreath product over $A$ and $\rad(\mathcal{A}_U)=e$. Then $\aut(\mathcal{A}_U)^{U/L}=\aut(\mathcal{A}_{U/L})$.
\end{lemm}

\begin{proof}
To prove the lemma it is sufficient to prove that  $\aut(\mathcal{A}_{U/L})$ is $2$-isolated. Indeed, the orbits of the componentwise action of the groups  $\aut(\mathcal{A}_{U/L})$ and  $\aut(\mathcal{A}_U)^{U/L}$ on $(U/L)^2$  coincide with the basic relations of the Cayley scheme corresponding to  $\mathcal{A}_{U/L}$. This implies that  $\aut(\mathcal{A}_{U/L})$ and  $\aut(\mathcal{A}_U)^{U/L}$ are $2$-equivalent. So if $\aut(\mathcal{A}_{U/L})$ is $2$-isolated then $\aut(\mathcal{A}_U)^{U/L}=\aut(\mathcal{A}_{U/L})$.

Since $\rad(\mathcal{A}_U)=e$, one of the statements of Lemma~\ref{circ} holds for  $\mathcal{A}_U$. If $\rk(\mathcal{A}_U)=~2$ then $U=L$ and obviously $\aut(\mathcal{A}_{U/L})$ is $2$-isolated. If one of  Statements  $2$-$4$ of Lemma~\ref{circ} holds for $\mathcal{A}_U$ then  $\mathcal{A}_{U/L}=\mathbb{Z}(U/L)$ or every basic set of $\mathcal{A}_{U/L}$ is of the form $\{x,x^{-1}\},~x\in U/L$. Therefore the stabilizer of  $L$ in  $\aut(\mathcal{A}_{U/L})$ has a faithful regular orbit and hence  $\aut(\mathcal{A}_{U/L})$ is $2$-isolated by \cite[Lemma~8.2]{MP3}.
\end{proof}

Note that the above lemma does not hold  for cyclic  $p$-groups, where $p>3$.

In \cite{EP3} it was proved that every $S$-ring over a cyclic $p$-group, where $p$ is a prime, is separable with respect to $\mathcal{K}_C$. Further we prove that all $S$-rings over cyclic  $2$- and $3$-groups are separable with respect to $\mathcal{K}_A$.

\begin{lemm}\label{sepcirc}
Let $\mathcal{A}$ be an  $S$-ring over  $A$. Then $\mathcal{A}$ is separable with respect to  $\mathcal{K}_A$.
\end{lemm}

\begin{proof}
We proceed by induction on $k$. If $k=1$ then the statement of the lemma holds because  $\rk(\mathcal{A})=2$ or $\mathcal{A}=\mathbb{Z}A$. Now let $k \geq 2$. If $\rad(\mathcal{A})=e$ then $\mathcal{A}$ is separble by Lemma~\ref{circ}. Suppose that $\rad(\mathcal{A}) > e$. Then from Lemma~\ref{circrad} it follows that $\mathcal{A}$ is the proper  $U/L$-wreath product for some $\mathcal{A}$-section $U/L$ such that $\rad(\mathcal{A}_U)=e$. By the induction hypothesis $S$-rings $\mathcal{A}_U$ and $\mathcal{A}_{A/L}$ are separable with respect to  $\mathcal{K}_A$. Lemma~\ref{circaut} yields that $\aut(\mathcal{A}_U)^{U/L}=\aut(\mathcal{A}_{U/L})$. Thus all conditions of Lemma~\ref{sepwr} hold for $\mathcal{A}$ and  hence  $\mathcal{A}$ is separable with respect to  $\mathcal{K}_A$.
\end{proof}

\section{$S$-rings over  $C_p\times C_{p^k}$}

Let $p\in\{2,3\}$ and $k\geq 1$. Put $D=A\times B$, where $A=\langle a\rangle,~|a|=p^k,~B=\langle b\rangle,~|b|=p$. Let $a_1=a^{p^{k-1}}$ and $a_2=a^{p^{k-2}}$. If $l\leq k$ then denote the subgroups $\{g\in A:|g|\leq p^l\}$ and  $\{g\in D:|g|\leq p^l\}$ of $D$ by $A_l$ and $D_l$ respectively. In these notations $A=A_k$ and $D=D_k$.

In the next lemma we describe sections of  $D$ such that $D$ is determined up to isomorphism  by these sections.

\begin{lemm}\label{group}
Let $q$ be a prime, $m\geq 3$, and $D^{'}$ be an abelian group of order  $q^{m+1}$. Suppose that the following conditions hold:

$1)$ $D^{'}$ contains at least two subgroups  of order  $q^{m-1}$ and one of them, say $A^{'}$, is cyclic$;$

$2)$ $A^{'}$ contains a subgroup $A_1^{'}$ of order $q$ such that $D^{'}/A_1^{'}$ is isomorphic to $C_q\times C_{q^{m-1}}.$

Then $D^{'}$ is isomorphic to  $C_q\times C_{q^{m}}$ or $C_{q^2}\times C_{q^{m-1}}$. Moreover, if $m \geq 4$ or $D^{'}$ contains a noncyclic subgroup  $W^{'}$ of order $q^2$ such that $|W^{'}\cap A^{'}|=q$ and $D^{'}/W^{'}$ is cyclic then $D^{'} \cong C_q\times C_{q^{m}}$.
\end{lemm}

\begin{proof}
Since  $D^{'}$ is abelian, it is the direct product of cyclic groups. Moreover, $D^{'}$ is isomorphic to one of the following groups
$$C_{q^{m+1}},C_q\times C_{q^{m}},C_{q^2}\times C_{q^{m-1}},C_q\times C_q \times C_{q^{m-1}}$$
because $A^{'}$ is the cyclic group of order $q^{m-1}$. Note that $D^{'}$ is noncyclic because  $D^{'}$ contains at least two subgroups of order  $q^{m-1}$. Suppose that $D^{'}\cong C_q\times C_q \times C_{q^{m-1}}$. Then $D^{'}=H^{'}\times A^{'}$, where  $H^{'}\cong C_q\times C_q$. If $A_1^{'}\leq A^{'}$ has order $q$ then  $D^{'}/A_1^{'}$ contains a subgroup isomorphic to $C_{q}\times C_{q}\times C_{q}$ because $m\geq 3$. We obtain a contradiction with $D^{'}/A_1^{'} \cong C_q\times C_{q^{m-1}}$.  Therefore $D^{'}\cong C_q\times C_{q^{m}}$ or $D^{'} \cong C_{q^2}\times C_{q^{m-1}}$.

If  $X\subseteq G \times H$ then denote the projections of  $X$ on $G$ and $H$ by $\pr_{G}(X)$ and
$\pr_{H}(X)$ respectively. Prove the second part of the lemma. Suppose that  $D^{'}=H^{'}\times A^{'}$, where $H^{'} \cong C_{q^2}$. If $m\geq 4$ then $D^{'}/A_1^{'}$ contains a subgroup isomorphic to  $C_{q^2}\times C_{q^2}$ and we obtain a contradiction with  $D^{'}/A_1^{'} \cong C_q\times C_{q^{m-1}}$. If  there is a noncyclic subgroup $W^{'}$ of order $q^2$ in $D^{'}$  such that $|W^{'}\cap A^{'}|=q$ and $D^{'}/W^{'}$ is cyclic then  $|\pr_{H^{'}}(W^{'})|=q$ because $|W^{'}\cap A^{'}|=q$. Since $W^{'}$ is noncyclic, we have $|\pr_{A^{'}}(W^{'})|=q$. So $W^{'}=\pr_{H^{'}}(W^{'})\times \pr_{A^{'}}(W^{'})\cong C_q\times C_q$. This implies that  $D^{'}/W^{'}$  is noncyclic, a contradiction with the assumption of the lemma. Thus  $D^{'}\cong C_{q}\times C_{q^{m}}$.
\end{proof}

Let $\mathcal{A}$ be an $S$-ring over $D$.  A basic set $X\in \mathcal{S}(\mathcal{A})$ is called   \emph{highest} if it contains an element of order $p^k$. By the \emph{radical} of $\mathcal{A}$ we mean the subgroup $\rad(\mathcal{A})$ generated by the subgroups $\rad(X)$, where $X$ runs over all highest basic sets of $\mathcal{A}$. A subset of  $D$ is called \emph{regular} if it consists of elements of the same order.

The description of all  $S$-rings over $D$ was obtained  for $p=2$ in \cite{MP3} and for $p=3$ in \cite{Ry}.

\begin{lemm}\label{element}
If $\mathcal{A}$ is an $S$-ring over $D$ and $k=1$ then one of the following statements holds:

$1)$ $\rk(\mathcal{A})=2;$

$2)$ $\mathcal{A}$ is the tensor product of two  $S$-rings over cyclic groups of order $p;$

$3)$ $\mathcal{A}$ is the wreath product of two  $S$-rings over cylic groups of order $p;$

$4)$ $p=3$ and $\mathcal{A}=\cyc(K,D)$, where $K=\{e,~\delta\},~\delta:x\rightarrow x^{-1};$

$5)$  $p=3$ and $\mathcal{A}\cong_{\cay} \cyc(K,D)$, where $K=\langle \sigma \rangle$ and $\sigma:(a_1,b)\rightarrow(b,a_1^2).$

\end{lemm}

\begin{proof}
Follows from the computer calculations that made by using the package COCO2P  \cite{GAP}.
\end{proof}

\begin{lemm}\label{Sring}
Let $\mathcal{A}$ be an $S$-ring over $D$ and $k\geq 2$. Then one of the following statements holds:

$1)$ $\rad(\mathcal{A})=e$ and there exist $\mathcal{A}$-subgroups $L,H\leq D$ such that  $\mathcal{A}=\mathcal{A}_H\otimes\mathcal{A}_L$, $\rk(\mathcal{A}_H)=2$, and $|L|\leq p \leq |H|;$

$2)$  $\rad(\mathcal{A})>e$ and there exist an $\mathcal{A}$-section $U/L$ such that $\mathcal{A}$ is the proper $U/L$-wreath product. Moreover, $\mathcal{A}_{U/L}=\mathbb{Z}(U/L)$, or $|U/L|\leq 4$, or $\rad(\mathcal{A}_U)=e$ and $|L|=p;$

$3)$ $\rad(\mathcal{A})=e$ and  $\mathcal{A}\cong_{\cay}\cyc(K,D)$, where $K\leq \operatorname{Aut}(D)$ is one of the groups listed in Table~$1$ for $p=2$ and in Table~$2$ for $p=3.$

\end{lemm}

\begin{center}

{\small
\begin{tabular}{|l|l|l|l|}
  \hline
  group & generators & order & $k$    \\
  \hline
  $K_0$ & $(a,b)\rightarrow (a,b)$ & $1$ & $k\geq 2$\\ \hline
  $K_1$ & $(a,b)\rightarrow (a^{-1},b)$  & $2$ & $k\geq 3$ \\  \hline
  $K_2$ & $(a,b)\rightarrow (a_1a^{-1},b)$  & $2$ & $k\geq 3$ \\  \hline
	$K_3$ & $(a,b)\rightarrow (a^{-1},ba_1)$  & $2$ & $k\geq 3$\\  \hline
	$K_4$ & $(a,b)\rightarrow (a_1a^{-1},ba_1)$  & $2$  & $k\geq 3$\\  \hline
	$K_5$ & $(a,b)\rightarrow (ba_2a,ba_1),~(a,b)\rightarrow (a^{-1},b)$ & $4$  & $k\geq 4$\\ \hline
  $K_6$ & $(a,b)\rightarrow (ba_2a,ba_1),~(a,b)\rightarrow (a_1a^{-1},b)$ & $4$  & $k\geq 4$\\ \hline
	$K_7$ & $(a,b)\rightarrow (ba^{-1},b)$  & $2$  & $k\geq 4$\\  \hline
	$K_8$ & $(a,b)\rightarrow (ba_1a^{-1},b)$  & $2$  & $k\geq 4$\\  \hline
	$K_9$ & $(a,b)\rightarrow (ba_2a,ba_1)$  & $2$  & $k\geq 3$\\  \hline
	$K_{10}$ & $(a,b)\rightarrow (ba_2a^{-1},ba_1)$  & $2$  & $k\geq 4$\\  \hline

\end{tabular}
}

Table 1.

\end{center}

	\begin{center}

{\small
\begin{tabular}{|l|l|l|l|}
  \hline
   group & generators & order & $k$   \\
  \hline
  $K_0$ & $(a,b)\rightarrow (a,b)$ & $1$  & $k\geq 2$\\ \hline
  $K_1$ & $(a,b)\rightarrow (a,b^2)$  & $2$  & $k\geq 2$\\  \hline
  $K_2$ & $(a,b)\rightarrow (a^{-1},b)$  & $2$ & $k\geq 2$\\  \hline
  $K_3$ & $(a,b)\rightarrow (a^{-1},b),~(a,b)\rightarrow (a,b^2)$ & $4$ & $k\geq 2$\\ \hline
  $K_4$ & $(a,b)\rightarrow (a^{-1},b^2)$ & $2$  & $k\geq 2$\\ \hline
  $K_5$ & $(a,b)\rightarrow (ba^{-1},b)$  & $2$  & $k\geq 2$\\  \hline
  $K_6$ & $(a,b)\rightarrow (ba,ba_1)$  & $3$  & $k\geq 3$\\  \hline
  $K_7$ & $(a,b)\rightarrow (ba,ba_1),~(a,b)\rightarrow (a,b^2a_1)$  & $6$  & $k\geq 3$\\  \hline
	$K_8$ & $(a,b)\rightarrow (ba,ba_1^2),~(a,b)\rightarrow (a^{-1},ba_1)$  & $6$ & $k\geq 3$\\  \hline
  $K_9$ & $(a,b)\rightarrow (ba,ba_1^2),~(a,b)\rightarrow (a^{-1},b^2)$  & $6$  & $k\geq 3$\\ \hline
	
\end{tabular}
}

Table 2.

\end{center}

\begin{proof}
The  lemma summarizes the statements of \cite[Theorem 6.1, Theorem 7.1, Theorem 9.1]{MP3} in case  $p=2$  and the statements of \cite[Theorem 4.1, Theorem 5.1, Theorem 6.1]{Ry} in case $p=3$.
\end{proof}

\begin{lemm}\label{aut}
Let  $\mathcal{A}$ be an $S$-ring over $D$ and Statement~$2$ of Lemma~$\ref{Sring}$ holds for $\mathcal{A}$. Then $\aut(\mathcal{A}_U)^{U/L}=\aut(\mathcal{A}_{U/L})$.
\end{lemm}

\begin{proof}

To prove the lemma we show that  $\aut(\mathcal{A}_{U/L})$ is $2$-isolated or $\aut(\mathcal{A}_{U/L})=\sym(U/L)=\aut(\mathcal{A}_U)^{U/L}$. In the former case   $\aut(\mathcal{A}_U)^{U/L}=\aut(\mathcal{A}_{U/L})$ because $\aut(\mathcal{A}_{U/L})$ and  $\aut(\mathcal{A}_U)^{U/L}$ are $2$-equivalent (see the proof of Lemma~\ref{circaut}).

If $\mathcal{A}_{U/L}=\mathbb{Z}(U/L)$ or $|U/L|\leq 4$ then it is obvious that $\aut(\mathcal{A}_{U/L})$ is $2$-isolated.  Further we  assume that   $\rad(\mathcal{A}_U)=e$ and $|L|=p$. Suppose that $U$ is cyclic. Then $L=A_1$ is the unique $\mathcal{A}$-subgroup of order  $p$ and one of the statements of Lemma~\ref{circ} holds for  $\mathcal{A}_U$. If $\rk(\mathcal{A}_U)=2$ then  $U=L$ and hence  $\aut(\mathcal{A}_{U/L})$ is $2$-isolated. If one of  Statements $2$-$4$ of Lemma~\ref{circ} holds  for  $\mathcal{A}_U$ then $\mathcal{A}_{U/L}=\mathbb{Z}(U/L)$ or every basic set of   $\mathcal{A}_{U/L}$ is of the form $\{x,x^{-1}\},~x\in U/L$. So the stabilizer of  $L$ in  $\aut(\mathcal{A}_{U/L})$ has a faithful regular orbit and  $\aut(\mathcal{A}_{U/L})$ is $2$-isolated by \cite[Lemma~8.2]{MP3}.

Suppose now that   $U$ is noncyclic. Then $U\cong D_l$ for some $l\leq k$. If $\mathcal{A}_U$ is regular then   $\aut(\mathcal{A}_{U/L})$ is $2$-isolated by \cite[Theorem 8.1]{MP3} for $p=2$ and by \cite[Corollary 5.2]{Ry} for $p=3$. If $\mathcal{A}_U$ is nonregular then Lemma~\ref{Sring} implies that $\mathcal{A}_U=\mathcal{A}_H\otimes\mathcal{A}_L$, where $\rk(\mathcal{A}_H)=2$. Note that $\rk(\mathcal{A}_L)=2$ or $\mathcal{A}_L=\mathbb{Z}L$ since $|L|=p$ and $p\in\{2,3\}$. If $\rk(\mathcal{A}_L)=2$ then 
$$\sym(U/L)\geq \aut(\mathcal{A}_{U/L}) \geq \aut(\mathcal{A}_U)^{U/L}=(\sym(H)\times \sym(L))^{U/L}=\sym(U/L);$$
if $\mathcal{A}_L=\mathbb{Z}L$ then
$$\sym(U/L)\geq \aut(\mathcal{A}_{U/L}) \geq \aut(\mathcal{A}_U)^{U/L}=(\sym(H)\times L_{right})^{U/L}=\sym(U/L).$$
Thus in both  cases   $\aut(\mathcal{A}_{U/L})=\sym(U/L)=\aut(\mathcal{A}_U)^{U/L}$.
\end{proof}

\section{Proof of Theorem~$1$}

All notations from the previous section are valid throughout this section.

Let $\mathcal{A}$ be an arbitrary  $S$-ring over $D$. Let us prove that $\mathcal{A}$ is separable with respect to  $\mathcal{K}_A$.  From now on throught this section we write for short ``separable''   instead ``separable with respect to $\mathcal{K}_A$''. Let $\mathcal{A}^{'}$ be an $S$-ring over an abelian group  $D^{'}$ and $\varphi:\mathcal{A}\rightarrow \mathcal{A}^{'}$ be an algebraic isomorphism. We proceed by induction on  $k$. Let $k=1$. Then one of the statements of Lemma~\ref{element} holds for $\mathcal{A}$. If $\rk(\mathcal{A})=2$ then, obviously, $\mathcal{A}$ is separable. If $\mathcal{A}$ is the tensor product or the wreath product of two $S$-rings over cyclic groups of order  $p$ then $\mathcal{A}$ is separable by Lemma~\ref{schurtens}. If Statement~$4$ of Lemma~\ref{element} holds for  $\mathcal{A}$ then $\mathcal{A}$ satisfies the conditions of Lemma~\ref{quasithin} and hence $\mathcal{A}$ is separable. 

Now suppose that Statement~$5$ of Lemma~\ref{element} holds for $\mathcal{A}$. Then $|D^{'}|=|D|=9$ and $\rk(\mathcal{A}^{'})=\rk(\mathcal{A})=3$. From $(3)$ it follows that $\rad(\mathcal{A}^{'})$ is trivial since $\rad(\mathcal{A})=e$. If $D^{'}$  is cyclic then   Lemma~\ref{circ} yields that  $\rk(\mathcal{A}^{'})=2$ or $\mathcal{A}^{'}$ and  $\mathcal{A}$ are quasi-thin that is not true. So $D^{'}$  is noncyclic and hence $D^{'}\cong D$. Further we assume that  $D^{'}=D$. From Lemma~\ref{element} it follows that  $\mathcal{A}$ is the unique up to Cayley isomorphism  $S$-ring of rank~$3$ over $D$ with basic sets of cardinalities~$1,4,4$. Therefore $\mathcal{A}^{'}\cong_{\cay}\mathcal{A}$. Let $X=\{x,x^{-1},y,y^{-1}\}$  be a nontrivial basic set of $\mathcal{A}$ and $X^{\varphi}=\{x^{'},(x^{'})^{-1},y^{'},(y^{'})^{-1}\}$. Then the Cayley isomorphism 
 $$\sigma:(x,y)\rightarrow (x^{'},y^{'})\in \aut(D)$$
 induces $\varphi$.

Now let $k\geq 2$. Then one of  Statements  $1$-$3$ of Lemma~\ref{Sring} holds for  $\mathcal{A}$. Every $S$-ring over the group of order~$p$, where $p\in\{2,3\}$, is separable. So if  $\mathcal{A}=\mathcal{A}_H\otimes\mathcal{A}_L$, where $\rk(\mathcal{A}_H)=2$ and $|L|\leq p$, then $\mathcal{A}$ is of rank~$2$ and hence separable or $\mathcal{A}$ is separable by Lemma~\ref{schurtens}.

Suppose that Statement~$2$ of Lemma~\ref{Sring} holds for $\mathcal{A}$.  From Lemma~\ref{aut} it follows that  $\aut(\mathcal{A}_U)^{U/L}=\aut(\mathcal{A}_{U/L})$. So by Lemma~\ref{sepwr} it is sufficient to prove a separability of   $\mathcal{A}_U$ and $\mathcal{A}_{D/L}$. If $U=D_l$ for some $l<k$ then $\mathcal{A}_U$ is separable by the induction hypothesis. If $U$ is cyclic then  $\mathcal{A}_U$ is separable by Lemma~\ref{sepcirc}. Similarly, $\mathcal{A}_{D/L}$ is separable by the induction hypothesis whenever $D/L$ is noncyclic and by Lemma~\ref{sepcirc} whenever $D/L$ is cyclic.

Suppose that Statement~$3$ of Lemma~\ref{Sring} holds for $\mathcal{A}$. Then $\mathcal{A}\cong_{\cay} \cyc(K,D)$, where $K\leq \operatorname{Aut}(D)$ is one of the groups listed in Table~$1$ for $p=2$ and in Table~$2$ for $p=3$. If $K=K_0$ then $\mathcal{A}=\mathbb{Z}D$ is separable. If $p=2$ and $K\in \{K_1,K_2,K_3,K_4,K_7,K_8,K_9,K_{10}\}$ or $p=3$ and $K\in \{K_1,K_2,K_4,K_5\}$ then $\mathcal{A}$ is quasi-thin and it is easy to check directly that there are no $\mathcal{A}$-subgroups $H$ in $D$ such that $H\cong C_2\times C_2$ and  $\mathcal{A}_{D/H}=\mathbb{Z}(D/H)$. So in these cases  $\mathcal{A}$ is separable by Lemma~\ref{quasithin}. If $p=3$ and $K=K_3$ then $\mathcal{A}$ is the tensor product of two quasi-thin   $S$-rings over cyclic  $3$-groups and hence $\mathcal{A}$ is separable by Lemma~\ref{quasithin} and Lemma~\ref{schurtens}. 

Consider the remaining cases. Let $p=2$ and $K\in \{K_5,K_6\}$ or $p=3$  and $K\in \{K_6,K_7,K_8,K_9\}$.

\begin{lemm}\label{isom}
$D^{'}\cong D$.
\end{lemm}

\begin{proof}
Let us check that $D^{'}$ satisfies the conditions of  Lemma~\ref{group}. At first consider the case $p=2$. 

1. The inequality    $k\geq 4$  holds because $K\in \{K_5,K_6\}$ (see Table~$1$).

2. Note that $\{bu,bu^{-1}\}\in \mathcal{S}(\mathcal{A})$ for every  $u\in A_{k-1}\setminus A_{k-2}$ as $K\in \{K_5,K_6\}$. Choose a basic set $Y\subseteq b(A_{k-1}\setminus A_{k-2})$. The group $F=\langle Y \rangle$ is a cyclic $\mathcal{A}$-subgroup of order $2^{k-1}$ and $\mathcal{A}_{F}=\cyc(M,F)$, where $M=\{\varepsilon,\sigma\}$ and $\sigma:x\rightarrow x^{-1}$. Since $k\geq 4$, we conclude that   $|F|>4$. Clearly that $\varphi$ induces the algebraic isomorphism 
$$\varphi_{F}:\mathcal{A}_{F}\rightarrow \mathcal{A}_{F^{\varphi}}.$$
From Lemma~\ref{simcycl} it follows that $F^{\varphi}$ is a cyclic subgroup of  $D^{'}$ of order $2^{k-1}$.

3. The group $D_{k-2}$ is an $\mathcal{A}$-subgroup of order  $2^{k-1}$ distinct from $F$. So  $D_{k-2}^{\varphi}$ is an $\mathcal{A}^{'}$-subgroup of order~$2^{k-1}$ distinct from  $F^{\varphi}$.

4. The group $A_1$ is an $\mathcal{A}$-subgroup of order~$2$. So  $A_1^{\varphi}$ is an $\mathcal{A}^{'}$-subgroup of order~$2$. Let $\pi:D\rightarrow D/A_1$ be the  quotient epimorphism and $X$ be a highest basic set of  $\mathcal{A}$. Then $X=\{x,x^{-1},ba_2x,ba_2^{-1}x^{-1}\}$ whenever $K=K_5$ and 
$X=\{x,a_1x^{-1},ba_2x,ba_2x^{-1}\}$ whenever $K=K_6$ for some generator $x$ of $A$. The set $\pi(X)$ is a generating set of $D/A_1$ and the following properties hold
$$|\pi(X)|=4,~\pi(X)=\pi(X)^{-1},~|\rad(\pi(X))|=2.~\eqno(4)$$ 
Let
$$\varphi_{D/A_1}:\mathcal{A}_{D/A_1}\rightarrow \mathcal{A}_{D^{'}/A_1^{\varphi}}$$
be the algebraic isomorphism induced by $\varphi$. From $(3)$ it follows that $\pi(X)^{\varphi_{D/A_1}}$ is a generating set of $D^{'}/A_1^{\varphi}$ and  $(4)$ also holds for  $\pi(X)^{\varphi_{D/A_1}}$. Let $\pi(X)^{\varphi_{D/A_1}}=\{x^{'},b^{'}x^{'},y^{'},b^{'}y^{'}\}$, where $\{e,b^{'}\}=\rad(\pi(X)^{\varphi_{D/A_1}})$. If $(x^{'})^{-1}=bx^{'}$ then $(x^{'})^2=(y^{'})^2=b^{'}$ and hence $|D^{'}/A_1^{\varphi}|=8$. So $|D|=|D^{'}|=16$. We obtain a contradiction because  $k\geq 4$ and $|D|\geq 32$. Therefore we may assume that  $y^{'}=(x^{'})^{-1}$. This implies that  $D^{'}/A_1^{\varphi}$ is generated by at most two elements one of which has order~$2$. Note that $D^{'}/A_1^{\varphi}$ is noncyclic because it contains at least two subgroups $A_2^{\varphi}/A_1^{\varphi}$ and $D_1^{\varphi}/A_1^{\varphi}$ of order~$2$. We conclude that $D^{'}/A_1^{\varphi}\cong  C_2\times C_{2^{k-1}}$. Thus  $D^{'}\cong D \cong  C_2\times C_{2^{k}}$ by Lemma~\ref{group}. 

Now let us check that the conditions of Lemma~\ref{group} hold for $D^{'}$ whenever $p=3$.

1. Since $K\in \{K_6,K_7, K_7, K_8\}$,  the group $A_{k-1}$ is a cyclic  $\mathcal{A}$-subgroup of order~$3^{k-1}$. Moreover, $\mathcal{A}_{A_{k-1}}=\mathbb{Z}A_{k-1}$ whenever  $K\in\{K_6,K_7\}$ and $\mathcal{A}_{A_{k-1}}=\cyc(M,A_{k-1})$, where $M=\{\varepsilon,\sigma\},~\sigma:x\rightarrow x^{-1}$, whenever $K\in\{K_8,K_9\}$. Clearly that $\varphi$ induces the algebraic isomorphism 
$$\varphi_{A_{k-1}}:\mathcal{A}_{A_{k-1}}\rightarrow \mathcal{A}_{(A_{k-1})^{\varphi}}.$$
Lemma~\ref{simcycl} yields that $A_{k-1}^{\varphi}$ is a cyclic $\mathcal{A}^{'}$-subgroup of order~$3^{k-1}$.

2. The group $D_{k-2}^{\varphi}$ is an $\mathcal{A}^{'}$subgroup of order~$3^{k-1}$ distinct from $A_{k-1}^{\varphi}$.

3. Note that $A_1^{\varphi}$ is an $\mathcal{A}^{'}$-subgroup of order~$3$. Let $\pi:D\rightarrow D/A_1$ be the quotient epimorphism and $X$ be a highest basic set of $\mathcal{A}$. If $K\in\{K_6,K_7\}$ then $X=\{x,bx,b^2a_1x\}$ and if $K\in\{K_8,K_9\}$ then $X=\{x,x^{-1},bx,b^2x^{-1},b^2a_1^2x,ba_1x^{-1}\}$ for some generator $x$ of $A$. The set  $\pi(X)$  is a generating set of  $D/A_1$,
$$|\pi(X)|=3,~|\rad(\pi(X))|=3,~\eqno(5)$$
whenever $K\in\{K_6,K_7\}$, and
$$|\pi(X)|=6,~\pi(X)=\pi(X)^{-1},~|\rad(\pi(X))|=3,~\eqno(6)$$ 
whenever $K\in\{K_8,K_9\}$. 
Let 
$$\varphi_{D/A_1}:\mathcal{A}_{D/A_1}\rightarrow \mathcal{A}_{D^{'}/A_1^{\varphi}}$$
be the algebraic isomorphism induced by  $\varphi$. From the properties of an algebraic isomorphism it follows that  $\pi(X)^{\varphi_{D/A_1}}$ is a generating set of  $D^{'}/A_1^{\varphi}$, $(5)$  holds for $\pi(X)^{\varphi_{D/A_1}}$ if $K\in\{K_6,K_7\}$, and $(6)$  holds for $\pi(X)^{\varphi_{D/A_1}}$ if $K\in\{K_8,K_9\}$. Since $k\geq 3$, we conclude that $\pi(X)^{\varphi_{D/A_1}}=x^{'}B^{'}$ or $\pi(X)^{\varphi_{D/A_1}}=x^{'}B^{'}\cup (x^{'})^{-1}B^{'}$, where $B^{'}=\rad(\pi(X)^{\varphi_{D/A_1}})$. Therefore $D^{'}/A_1^{\varphi}$ is generated by at most two elements one of which has order~$3$. Note that $D^{'}/A_1^{\varphi}$ is noncyclic because it contains at least two subgroups $A_2^{\varphi}/A_1^{\varphi}$ and $D_1^{\varphi}/A_1^{\varphi}$ of order~$3$. This implies that $D^{'}/A_1^{\varphi}\cong  C_3\times C_{3^{k-1}}$.

By the first part of Lemma~\ref{group} the group  $D^{'}$ is isomorphic to $C_3\times C_{3^{k}}$ or $C_{9}\times C_{3^{k-1}}$. If $k\geq 4$ then $D^{'}\cong C_3\times C_{3^{k}}$ by the second part of Lemma~\ref{group}. Now let $k=3$. Put $D_1^{'}=\{x\in D^{'}:|x|=3\}$. Clearly that $|D_1^{'}|=9$. Suppose thay $D_1^{'}$ is an $\mathcal{A}^{'}$-subgroup. Then $D_1^{'}=D_1^{\varphi}$ or $D_1^{'}=A_2^{\varphi}$ since only the groups $D_1$ and $A_2$ are $\mathcal{A}$-subgroups of order~$9$. However, $A_2^{\varphi}$ is cyclic by Lemma~\ref{simcycl}. So  $D_1^{'}=D_1^{\varphi}$. The group $D/D_1$ is cyclic, $\mathcal{A}_{D/D_1}=\mathbb{Z}(D/D_1)$ or $\mathcal{A}_{D/D_1}=\cyc(M,D/D_1)$, where $M=\{\varepsilon,\sigma\},~\sigma:x\rightarrow x^{-1}$. Therefore the group $D^{'}/D_1^{'}$ is also cyclic by Lemma~\ref{simcycl}.  Note that $|D_1\cap A_{2}|=3$ and hence $|D_1^{'}\cap A_{2}^{\varphi}|=3$. Thus  $D^{'}\cong D\cong C_3\times C_{9}$ by the second part of Lemma~\ref{group}.

Suppose that $D^{'} \cong C_{9}\times C_9$. Then by the above discussion $D_1^{'}$ is not an  $\mathcal{A}^{'}$-subgroup. The  group $D_2$ is an $\mathcal{A}$-subgroup as  $\mathcal{A}$ is regular. So $D_2^{\varphi}$ is an $\mathcal{A}^{'}$-subgroup and  $D^{'}\setminus~D_2^{\varphi}$ is an $\mathcal{A}^{'}$-set. Since $|D_2|=|D_2^{\varphi}|=27$, the inclusion $D_1^{'}\subset~D_2^{\varphi}$ holds. Let $X\subseteq D\setminus D_2$ be a highest basic set of $\mathcal{A}$. Then $|X|\in \{3,6\}$, $\rad(X)$ is trivial, $\langle X \rangle=D$, and if $|X|=6$ then $X=X^{-1}$. These properties also hold for  $X^{\varphi}$. 

Suppose that $|xD_1^{'}\cap X^{\varphi}|=3$, where $x^{'}\in X^{\varphi}$. If  $|X^{\varphi}|=3$ then $Y^{'}=X^{\varphi}(X^{\varphi})^{-1}$ is an $\mathcal{A}^{'}$-set and $Y^{'}\subseteq D_1^{'}$. Moreover, $Y^{'}\nsubseteq  A_1^{\varphi}$ since otherwise $\rad(X^{\varphi})=A_1^{\varphi}$. So $D_1^{'}=\langle  A_1^{\varphi},Y^{'} \rangle $ is  an  $\mathcal{A}^{'}$-subgroup, a contradiction. Let  $|X^{\varphi}|=6$. If $((x^{'})^2\cup (x^{'})^{-2})D_1^{'}\cap D_2^{\varphi}\neq \varnothing$ then $((x^{'})^2\cup (x^{'})^{-2})D_1^{'}\subset D_2^{\varphi}$ and  hence $x^{'}\in  D_2^{\varphi}$. On the other hand,  $x^{'}\in D^{'}\setminus D_2^{\varphi}$, a contradiction. Therefore $((x^{'})^2\cup (x^{'})^{-2})D_1^{'}\cap D_2^{\varphi}=\varnothing$. 
This implies that $Y^{'}=X^2\cap  D_2^{\varphi}$ is an $\mathcal{A}^{'}$-set and $Y^{'}\subseteq D_1^{'}$. Note that $Y^{'}\nsubseteq  A_1^{\varphi}$ because otherwise $\rad(X^{\varphi})=A_1^{\varphi}$. We conclude that $D_1^{'}=\langle  A_1^{\varphi},Y^{'} \rangle $ is an $\mathcal{A}^{'}$-subgroup, a contradiction. So  $|xD_1^{'}\cap X^{\varphi}|\neq 3$ for every highest basic set $X\in \mathcal{S}(\mathcal{A})$. Then $Y^{'}=(X^{\varphi})^{[3]}$ is an $\mathcal{A}^{'}$-set by Lemma~\ref{sch}. Since  $D^{'} \cong C_{9}\times C_9$, we obtain that $Y^{'}\subseteq D_1^{'}$. If $Y^{'}\nsubseteq A_1^{\varphi}$ then $D_1^{'}=\langle  A_1^{\varphi},Y^{'} \rangle $ is an $\mathcal{A}^{'}$-subgroup, a contradiction. Therefore $(X^{\varphi})^{[3]}\subseteq  A_1^{\varphi}$ for every highest basic set   $X$. The union of all highest basic sets of  $\mathcal{A}$ has cardinality~$54$. So $|\{x\in D^{'}:x^3\in A_1^{\varphi}\}|\geq 54$ that is impossible whenever  $D^{'} \cong C_{9}\times C_9$. Thus $D^{'}$is not isomorphic to  $C_{9}\times C_9$ and hence $D^{'}\cong D\cong C_3\times C_{9}$.
\end{proof}

Further we may assume without loss of generality that  $D=D^{'}$.

\begin{lemm}\label{cayis}
$\mathcal{A}^{'}\cong_{\cay} \mathcal{A}$.
\end{lemm}

\begin{proof}
One of the statements of Lemma~\ref{Sring} holds for  $\mathcal{A}^{'}$. Since $\rad(\mathcal{A})$ is trivial, from $(3)$ it follows that $\rad(\mathcal{A}^{'})$ is also trivial. So Statement~$2$ of Lemma~\ref{Sring} does not  hold for  $\mathcal{A}^{'}$. Clearly that  $|\mathcal{S}(\mathcal{A}^{'})|=|\mathcal{S}(\mathcal{A})|>6$. This implies that  Statement~$1$ of Lemma~\ref{Sring}  does not hold for $\mathcal{A}^{'}$. Thus  Statement~$3$ of Lemma~\ref{Sring} holds for $\mathcal{A}^{'}$. It means that  $\mathcal{A}^{'}\cong_{\cay}\cyc(K^{'},D)$, where $K^{'}\leq \operatorname{Aut}(D)$ is one of the groups listed in Table~1 for $p=2$ and in Table~2 for $p=3$.

At first consider the case  $p=2$. There are basic sets of $\mathcal{A}$ of cardinality~$4$. So $\mathcal{A}^{'}$ also has basic sets of cardinality~$4$. This yields that $\mathcal{A}^{'}$ is not quasi-thin. Therefore $K^{'}\in \{K_5,K_6\}$. Not that $\cyc(K_5,D)$ is symmetric whereas  $\cyc(K_6,D)$ is not symmetric. If $\mathcal{A}$ is symmetric then by the properties of an algebraic isomorphism  $\mathcal{A}^{'}$ is also symmetric and hence $\mathcal{A}^{'}\cong_{\cay}\cyc(K_5,D)\cong_{\cay}\mathcal{A}$. If $\mathcal{A}$ is not symmetric then  $\mathcal{A}^{'}$ is also not symmetric and  $\mathcal{A}^{'}\cong_{\cay}\cyc(K_6,D)\cong_{\cay}\mathcal{A}$.

Now let $p=3$. Given an $S$-ring $\mathcal{B}$ put 
$$\mathcal{N}(\mathcal{B})=\{|X|:X\in \mathcal{S}(\mathcal{B})\}.$$ 
It is clear that $\mathcal{N}(\mathcal{B})$ is  invariant under algebraic isomorphisms. Therefore the statement of  the lemma follows from the next observation: $\mathcal{B}=\cyc(K_i,D)$ is the unique up to Cayley isomorphism cyclotomic $S$-ring over  $D$ such that

$1)$  $\mathcal{N}(\mathcal{B})=\{1,3\}$ if $i=6$;

$2)$ $\mathcal{N}(\mathcal{B})=\{1,3,6\}$  if $i=7$;

$3)$ $\mathcal{N}(\mathcal{B})=\{1,2,3,6\}$ if $i=8$;

$4)$ $\mathcal{N}(\mathcal{B})=\{1,2,6\}$ if $i=9$.
\end{proof}

Let $X\in\mathcal{S}(\mathcal{A})$ be a highest basic set. If $p=3$ and $K\in\{K_6,K_8\}$ then put $Z=bA_1$ and  $Y=X\cup Z$; otherwise put $Y=X$.

\begin{lemm}\label{gen}
$\mathcal{A}=\langle \underline{Y} \rangle$ and $\mathcal{A}^{'}=\langle \underline{Y^{\varphi}} \rangle$.
\end{lemm}

\begin{proof}
 Put $\mathcal{A}_1=\langle \underline{Y} \rangle$. From Lemma~\ref{Sring} it follows that $X$ contains a generator $x$ of $A$ and 
$$X=\{x,x^{-1},ba_2x,ba_2^{-1}x^{-1}\},$$ 
if $p=2$ and $K=K_5$; 
$$X=\{x,a_1x^{-1},ba_2x,ba_2x^{-1}\},$$ 
if $p=2$ and $K=K_6$;
$$X=\{x,bx,b^2a_1x\},$$
if $p=3$ and $K\in\{K_6,K_7\}$;
$$X=\{x,x^{-1},bx,b^2x^{-1},b^2a_1^2x,ba_1x^{-1}\},$$
if $p=3$ and $K\in\{K_8,K_9\}$.

Statement~$1$ of Lemma~\ref{Sring} does not hold for  $\mathcal{A}_1$ because otherwise every element of order  $p^k$ must lie in a basic set of cardinality at least  $p^k-1$, where $k \geq 4$ if $p=2$ and $k \geq 3$ if $p=3$. It is easy to check that every subset of $Y$  consisting of elements of order $p^k$ has the trivial radical. So $\mathcal{A}_1$ has a highest basic set with the trivial radical  and Statement~$2$ of Lemma~\ref{Sring} does not hold for $\mathcal{A}_1$. Thus $\mathcal{A}_1$ is cyclotomic. From the classification of all cyclotomic  $S$-rings over $D$ that is given in Lemma~\ref{Sring} it follows that $\mathcal{A}_1=\mathcal{A}$.  It should be mentioned that if  $p=3$ then $\cyc(K_6,D)$ and $\cyc(K_7,D)$ as well as  $\cyc(K_8,D)$ and $\cyc(K_9,D)$ have the same highest basic sets. In the cases when $K\in\{K_7,K_9\}$ the $S$-ring $\mathcal{A}$ is generated by the highest basic set  $X$ and in the cases when $K\in\{K_6,K_8\}$ the $S$-ring $\mathcal{A}$ is generated by $X\cup Z$.

From~$(3)$  it follows that $X^{\varphi}$  is a highest basic set of $\mathcal{A}^{'}$. If $p=3$ then there are exactly two subgroups $A_2$ and $D_1$ of order~$9$ in $D$. These subgroups are $\mathcal{A}$-subgroups. The group $A_2^{\varphi}$ is a cyclic  $\mathcal{A}^{'}$-subgroup by Lemma~\ref{simcycl} and hence $A_2^{\varphi}=A_2$. So $D_1^{\varphi}=D_1$. If $K\in\{K_6,K_8\}$ then  $Z^{\varphi}\in \{Z,Z^{-1}\}$ because only  $Z$ and $Z^{-1}$ are basic sets of cardinality~$3$ inside $D_1$. Thus if $Y=X\cup Z$ then $Y^{\varphi}=X^{\varphi}\cup Z$ or $Y^{\varphi}=X^{\varphi}\cup Z^{-1}$. Since $\mathcal{A}^{'}\cong_{\cay} \mathcal{A}$, by using the above arguments it can  be proved that $\mathcal{A}^{'}=\langle \underline{Y^{\varphi}} \rangle$.
\end{proof}

\begin{lemm}\label{cayind}
The algebraic isomorphism $\varphi$ is induced by a Cayley isomorphism.
\end{lemm}

\begin{proof}

Lemma~\ref{cayis} implies that there exists a Cayley isomorphism $f$ from $\mathcal{A}$ to $\mathcal{A}^{'}$. The sets $X^{\varphi}$ and $X^{f}$ are highest basic sets of $\mathcal{A}^{'}$. If $p=2$ then every highest basic set of $\mathcal{A}^{'}$ is of the form
$$X_0\cup bX_1,$$
where $X_i\subseteq A,~X_i\neq \varnothing,~i\in \{0,1\}$, because $K\in\{K_5,K_6\}$. Similarly, if $p=3$ then every highest basic set of $\mathcal{A}^{'}$ is of the form
$$X_0\cup bX_1\cup b^2X_2,$$
where $X_i\subseteq A,~X_i\neq \varnothing,~i\in \{0,1,2\}$. So by Lemma~\ref{burn} the sets  $X^{\varphi}$ and $X^{f}$ are rationally conjugate and there exists a Cayley isomorphism  $g$ from $\mathcal{A}^{'}$ onto itself such that $X^{fg}=X^{\varphi}$. The Cayley isomorphism $fg$ from $\mathcal{A}$ to $\mathcal{A}^{'}$ induces the algebraic isomorphism  $\varphi_{fg}$. If $p=2$ or $p=3$ and $K\in\{K_7,K_9\}$ then $\mathcal{A}=\langle \underline{X} \rangle$ and $\mathcal{A}^{'}=\langle \underline{X^{\varphi}} \rangle$ by Lemma~\ref{gen}. In these cases from Lemma~\ref{uniq} it follows that $\varphi=\varphi_{fg}$.

Now let $p=3$ and $K\in\{K_6,K_8\}$. From Lemma~\ref{gen} it follows that $\mathcal{A}=\langle \underline{Y} \rangle$ and $\mathcal{A}^{'}=\langle \underline{Y^{\varphi}} \rangle$. If $Z^{fg}= Z^{\varphi}$ then $Y^{fg}=Y^{\varphi}$ and hence Lemma~\ref{uniq} implies that $\varphi=\varphi_{fg}$. Suppose that $Z^{fg}\neq Z^{\varphi}$. Without loss of generality we may assume that  $Z^{fg}=\{b,ba_1,ba_1^2\}$ and $Z^{\varphi}=\{b^2,b^2a_1,b^2a_1^2\}$. If $K=K_6$ then $X^{\varphi}=\{y,by,b^2a_1y\}$ or $X^{\varphi}=\{y,ba_1^2y,b^2y\}$; if $K=K_8$ then $X^{\varphi}=\{y,y^{-1},by,ba_1y^{-1},b^2a_1^2y,b^2y^{-1}\}$, where $y$ is a generator of $A$. Put
$$h:(y,b)\rightarrow (y,b^2a_1)\in \aut(D)$$
 if $K=K_6$ and 
$$h:(y,b)\rightarrow (y,b^2a_1^2)\in \aut(D)$$
if $K=K_8$. The straightforward check shows that $X^{fgh}=(X^{\varphi})^h=X^{\varphi}$ and $Z^{fgh}=Z^{\varphi}$. Since $X^{\varphi}$ is a highest basic set of the cyclotomic  $S$-ring $(\mathcal{A}^{'})^{h}$, we conclude that $(\mathcal{A}^{'})^{h}=\mathcal{A}^{'}$. Therefore $fgh$ is a Cayley isomorphism from $\mathcal{A}$ to $\mathcal{A}^{'}$ such that $Y^{fgh}=Y^{\varphi}$. From Lemma~\ref{uniq} it follows that $\varphi=\varphi_{fgh}$, where $\varphi_{fgh}$ is the algebraic isomorphism induced by  $fgh$.
\end{proof}

Thus if $\mathcal{A}=\cyc(K,D)$, where $K\in\{K_5,K_6\}$ whenever $p=2$ and $K\in\{K_6,K_7,K_8,K_9\}$ whenever $p=3$, every algebraic isomorphism of $\mathcal{A}$ is induced by a Cayley isomorphism. So $\mathcal{A}$ is separable and the proof of Theorem~$1$ is complete.

\section{Separability and the isomorphism problem for Cayley graphs}

 Let $\Gamma=\cay(G,X)$ and $\Gamma^{'}=\cay(G^{'},X^{'})$ be Cayley graphs over groups  $G$ and $G^{'}$  respectively. Denote the set of all isomorphisms from $\Gamma$ to $\Gamma^{'}$ by $\iso(\Gamma,\Gamma^{'})$. Fix classes of groups  $\mathcal{K}$ and $\mathcal{K}^{'}$. The isomorpism problem for Cayley graphs can be formulated as follows. 

\textbf{ISO.} Given Cayley graphs  $\Gamma$ over $G\in \mathcal{K}$ and $\Gamma^{'}$ over $G^{'}\in \mathcal{K}^{'}$  determine whether  $\iso(\Gamma,\Gamma^{'})\neq \varnothing$.

Further we consider the reduction of  ISO to the  following problem:

\textbf{ALISO.} Given Cayley schemes   $\mathcal{C}$ over  $G\in \mathcal{K}$ and $\mathcal{C}^{'}$  over $G^{'}\in \mathcal{K}^{'}$  and an algebraic isomorphism  $\varphi:\mathcal{C}\rightarrow \mathcal{C}^{'}$  determine whether  $\iso(\mathcal{C},\mathcal{C}^{'},\varphi)\neq \varnothing$.  

\begin{prop}\label{reduction}
 ISO  is reduced to ALISO in time $|G|^{O(1)}$.
\end{prop}

\begin{proof}
Suppose that there is an algorithm $Al_1$ solving ALISO. We assume that $|G|=|G^{'}|=n$ since otherwise, obviously,  $\Gamma$ and $\Gamma^{'}$ are not isomorphic. Denote the sets of edges of $\Gamma$ and $\Gamma^{'}$ by  $E$  and $E^{'}$ respectively. Let 
$$\mathcal{T}=(\diag(G\times G),E,G\times G\setminus (E\cup \diag(G\times G))$$ 
and
$$\mathcal{T}^{'}=(\diag(G^{'}\times G^{'}),E^{'},G^{'}\times G^{'}\setminus (E^{'}\cup \diag(G^{'}\times G^{'}))$$ 
be the corresponding to $\Gamma$ and $\Gamma^{'}$ ordered partitions of  $G\times G$ and $G^{'}\times G^{'}$. By using the Weisfeiler-Leman algorithm  (\cite{Weis,WeisL}) we can  construct in time $n^{O(1)}$ starting from $\mathcal{T}$ and $\mathcal{T}^{'}$  the ordered partitions $\mathcal{R}=(P_1,P_2,\ldots,P_k)$ and  $\mathcal{R}^{'}=(Q_1,Q_2,\ldots,Q_l)$ defining Cayley schemes $\mathcal{C}$ and $\mathcal{C}^{'}$ over $G$ and $G^{'}$ respectively. These schemes are the least schemes for which  $E$ and $E^{'}$ are unions of basic relations.

If $f\in \iso(\Gamma,\Gamma^{'})$ then by  properties of the Weisfeiler-Leman algorithm $k=l$, $f$ is an isomorphism from $\mathcal{C}$ to $\mathcal{C}^{'}$ such that $P_i^f=Q_i,~i=1,\ldots,k$, and hence the bijection $\varphi:P_i\rightarrow Q_i,~i=1,\ldots,k$, is an algebraic isomorphism. Conversly, if $\varphi:P_i\rightarrow Q_i$ is an algebraic isomorphism and $f\in \iso(\mathcal{C},\mathcal{C}^{'},\varphi)$ then $E^f=E^{'}$ and hence $f \in \iso(\Gamma,\Gamma^{'})$. Therefore $\iso(\mathcal{C},\mathcal{C}^{'},\varphi)=\iso(\Gamma,\Gamma^{'})$.

One can check in time $n^{O(1)}$ whether the mapping $\varphi:P_i\rightarrow Q_i,~i=1,\ldots,k$, is an algebraic isomorphism because  $\mathcal{C}$ has at most  $n^3$ intersection numbers. If $\varphi$ is not an algebraic isomorphism then  $\Gamma$ and $\Gamma^{'}$ are not isomorphic. If $\varphi$  is an algebraic isomorphism then applying $Al_1$ it can be  determined whether the set $\iso(\mathcal{C},\mathcal{C}^{'}, \varphi)=\iso(\Gamma,\Gamma^{'})$ is not empty.
\end{proof}

Now let $\mathcal{K}$ be the class of groups isomorphic to $D=C_p\times C_{p^k}$, where  $p\in \{2,3\}$ and $k\geq 1$, and $\mathcal{K}^{'}$ be the class of all abelian groups. If $\mathcal{C}$ is a Cayley scheme over $G\in \mathcal{K}$ that is separable with respect to the class of Cayley schemes over  groups from  $\mathcal{K}^{'}$ and $\mathcal{C}^{'}\in \mathcal{K}^{'}$ then ALISO is trivial because  for every algebraic isomorphism  $\varphi:\mathcal{C}\rightarrow \mathcal{C}^{'}$ the set $\iso(\mathcal{C},\mathcal{C}^{'},\varphi)$ is not empty. Therefore ISO can be solved in time $|G|^{O(1)}$. Thus Theorem~$2$ follows from Theorem~$1$, Proposition~\ref{connect}, and Proposition~\ref{reduction} applying to the classes $\mathcal{K}$ and $\mathcal{K}^{'}$.

It should be mentioned that the material of this section is based on the concepts suggested in \cite{Weis} and developed in  \cite{EP5}.

\end{document}